\documentclass[12pt]{amsart}
\usepackage{latexsym,fancyhdr,amssymb,color,amsmath,amsthm,graphicx,listings,comment}
\usepackage[section]{placeins}
\newtheorem{thm}{Theorem}

\let\paragraph\subsection

\title{Euler Characteristics of Random Manifolds}
\author{Oliver Knill}
\date{July 26, 2026}
\address{Department of Mathematics \\ Harvard University \\ Cambridge, MA, 02138 }


\begin{document}
\maketitle

\begin{abstract}
We prove that the expectation of the Euler characteristic $\chi(H)$ of random level surface
$H$ in a given simplicial complex $G$ is ${\rm E}[\chi(H)]=2-2K(G)-\chi(G)$, where
$K(G)=1-f_0(G)/2+f_1(G)/3-...$ is the curvature functional of $G$.
For odd-dimensional manifold $G$, we obtain random even-dimensional sub manifolds with 
Euler characteristic expectation ${\rm E}[\chi(H)]=2-2K(G)$. 
For $3$-manifolds for example, where $K(G)=1-f_0/2+f_1/3-f_2/4+f_3/5$ with
Dehn-Sommerville relations $f_3=2 f_2$ and $\chi(G)=f_0-f_1+f_2-f_3=0$ we have 
${\rm E}[\chi(H)]=f_1/3-f_2/5$. Our result also shows that the expectation
of the f-vector of a submanifold is explicitly linked to the f-vector of the host manifold. 
\end{abstract}

\section{Definitions}
\paragraph{}
A {\bf finite abstract simplicial complex} $G$ is a finite set of non-empty sets closed under 
the operation of taking non-empty subsets. 
The {\bf star} $U(x)=\{ y \in G, x \subset y\}$ is the smallest open set containing $x \in G$.
The {\bf closure} $\overline{A}$ of a subset $A \subset G$ is the smallest simplicial complex that contains
$A$. If $x \in G$, its {\bf unit sphere} is $S(x)= \overline{U(x)} \setminus U(x)$. 
The sets $\{ U(x)\}_{x \in G}$ define a basis for the {\bf Alexandroff topology} 
on $G$. It is classical finite topology but non-Hausdorff if the maximal dimension $q$ of 
$G$ is positive.

\begin{figure}[!htpb]
\scalebox{0.85}{\includegraphics{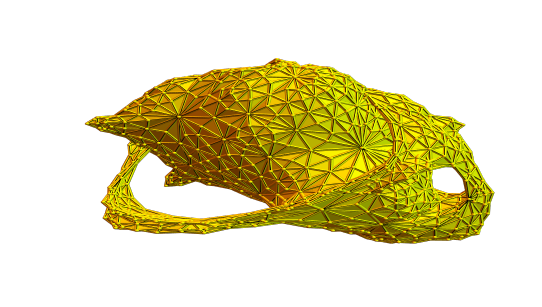}}
\scalebox{0.85}{\includegraphics{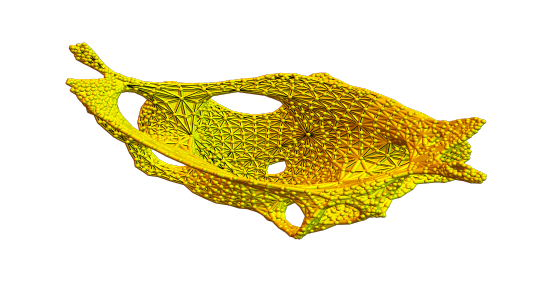}}
\label{Random manifold}
\caption{
To the left we see a random manifold in a $3$-sphere $G$.
To the right we see a random $2$-manifold with boundary that has been
chosen randomly in  $3$-ball $G$.
}
\end{figure}

\paragraph{}
With $\omega(x) = (-1)^{\rm dim}(x)$, the {\bf Euler characteristic} of $G$ is $\chi(G)=\sum_{x \in G} \omega(x)$. 
Inductively, $G$ is called a {\bf Dehn-Sommerville $q$-manifold} if for all $x \in G$, the unit sphere $S(x)$ is a Dehn-Sommerville 
$(q-1)$-manifold of Euler characteristic $\chi(S(x))=1-(-1)^q$. The induction starts with $0=\{\}$ being Dehn-Sommerville..
The {\bf join} $A \oplus B$ of two disjoint simplicial complexes
is $A \oplus B = A \cup B \cup \{ x \cup y, x \in A, y \in B\}$. Since $S_{A \oplus B}(x)=S(x) \oplus B$ if $x \in A$
and $S_{A \oplus B}(y) = A \oplus S(y)$ if $y \in B$, odd-dimensional Dehn-Sommerville manifolds form a sub-monoid
of the join monoid of all simplicial complexes. The {\bf zero element} is the {\bf empty complex} $0=\{\}$. It is the unique 
$(-1)$-dimensional Dehn-Sommerville complex and also called the {\bf void}. It is the initial object 
in the category.

\paragraph{}
Inductively, a complex $G$ is {\bf contractible} if there is $x \in G$ such that both $S(x)$ and $G \setminus U(x)$ are contractible
simplicial complexes. The {\bf $1$-point complex} $1=\{ 0\}$ is declared to be the smallest contractible complex. 
A complex $G$ is called a {\bf $q$-manifold} if every unit sphere
$S(x)$ is a $(q-1)$-sphere. A $q$-manifold is called a {\bf $q$-sphere} if there exists 
$x \in G$ such that $G \setminus U(x)$ is contractible. 
{\bf Euler's gem formula} $\chi(G)=1+(-1)^q$ for a $q$-sphere follows from the {\bf valuation property}
$\chi(A \cup B) + \chi(A \cap B) = \chi(A) + \chi(B)$ for Euler characteristic and the inductively established fact
that every contractible complex has $\chi(G)=1$ and that every unit ball $B(x) = \overline{U(x)}$ has $\chi(B(x))=1$.

\paragraph{}
Using definitions inductively with respect to dimension $q$, one can see that 
all $q$-manifolds are Dehn-Sommerville $q$-manifolds. The later class shares many
properties of manifolds but being Dehn-Sommerville can be verified faster than establishing $G$ to be a manifold.
Every odd-dimensional Dehn-Sommerville manifold is also a Dehn-Sommerville sphere. The suspension of an 
even-dimensional torus is an example of a Dehn-Sommerville manifold that is not a manifold.  
A complex $G$ of is a {\bf q-variety} if its maximal dimension is $q$ and every $S(x)$ is a $(q-1)$ variety, starting
the definition with the assumption that $0$ is a $(-1)$ variety.  
Varieties with join operation form a sub-monoid of all complexes. 

\paragraph{}
A simplicial complex $G$ defines a {\bf graph of $G$} $\Gamma(G)$, in which $G$ are the vertices and two simplices $x,y \in G$ are
connected by an edge if $x \subset y$ or $y \subset x$.  
The {\bf stable sphere} $S^-(x) = \{ y \in G | y \subset x, y \neq x\}$ is the boundary complex of the simplex $x$.
It is a simplicial complex. If $x+y$ means the {\bf symmetric difference}, the {\bf unstable sphere} 
$S^+(x) = \{ y + x | y \in G, x \subset y, y \neq x\}$ is a simplicial complex too. 
Compare with $U^+(x) = \{ y \in G | x \subset y, y \neq x\} = U(x) \setminus \{x\}$ which is an open set. 
For any simplicial complex $G$, the {\bf hyperbolicity relation} $S(x) = S^-(x) \oplus S^+(x)$ holds. 

\paragraph{}
The {\bf Whitney complex} of a graph $\Gamma$ consists of the vertex sets of complete subgraphs of $\Gamma$. 
The Barycentric refinement $\phi(A)$ of a subset $A$ of $G$ is the Whitney complex of the graph $\Gamma(A)$. 
Dehn-Sommerville spaces have the property that the {\bf hyperbolic decomposition} 
$S(x) = S^-(x) \oplus S^+(x)$ decomposes $S(x)$ into spaces $S^-(x),S^+(x)$ that are both Dehn-Sommerville spheres.
In particular, if $G$ is a manifold, then both $S^-(x)$ and $S^+(x)$ are spheres and $S(x)$ is the join of two spheres:
if  ${\rm dim}(x)=k$, then $S^-(x)$ is a $(k-1)$-sphere and $S^+(x)$ is a $(q-k-1)$-sphere which
join to the $(q-1)$-sphere $S(x)= S^-(x) \oplus S^+(x)$. 

\paragraph{}
From the hyperbolic property follows by induction the {\bf level surface theorem},
telling that for a Dehn-Sommerville $q$-manifold,
any function $g: V(G) \to K_k=\{0, \dots, k\}$ has a {\bf level set} $G_g = \{ y \in G, f(y) = K_k \}$ that is either
empty or a Dehn-Sommerville $(q-k)$-manifold in the sense that its Barycentric refinement is. 
If $G$ was a $q$-manifold, then $G_g$ is either empty of a $(q-k)$-manifold. If $G$ is a $q$-variety, then $G_g$ is
either empty or a $(q-k)$-variety. 
For computations it is not necessary to do the Barycentric refinement of $G_g$. Just adjust 
the dimension functional, so that $k$-simplices in $G_g$ are the vertices. Technically we deal 
then with {\bf delta sets}. Like the formation of Cartesian products or quotients, also the level-set operation 
lands in the delta set category. One can then look at its graph to get back a simplicial complex. 

\paragraph{}
A {\bf valuation} $X$ is a map on sub-simplicial complexes of $G$ that satisfies
$X(A \cup B) + X(A \cap B) = X(A) + X(B)$ for all $A,B$. Hadwiger's theorem reads 
that the linear space of valuation of a $q$-dimensional simplicial complex $G$ has dimension $q+1$.
The valuations $f_k(G) = |\{ x \in G | {\rm dim}(x)=k\}|$ form a basis in the linear space of valuations.
The {\bf f-vector} $f=(f_0, \dots, f_q)$  or the {\bf simplex generating function} 
$f_G(t) = 1+f_0 t+ \cdots + f_q t^{q+1}$ both encode a valuation faithfully; they do so
either as a vector or as a polynomial.

\paragraph{}
The simplex generating polynomial satisfies $f_{G \oplus H}(t) = f_G(t) f_H(t)$ when taking joins. From the
{\bf functional Gauss-Bonnet formula} $f'_G(t) = \sum_{v \in V} f_{S(v)}(t)$ follows by induction that $f_G(-1-t)= \pm f_G(t)$ for Dehn-Sommerville
manifolds, implying that the polynomials $h(t)=(t-1)^q f(1/(t-1))$ have palindromic coefficients, explaining 
why half of the combinatorial data are redundant for Dehn-Sommerville manifolds and especially for manifolds. 
The corresponding Dehn-Sommerville identities allow to simplify expressions involving $f$. Special
cases are $\chi(G)=0$ for odd-dimensional Dehn-Sommerville manifolds or that $2 f_{q-1} = (q+1) f_q$, encoding that every
$(q-1)$-simplex contains exactly $2$ maximal simplices and that every $q$-simplex contains exactly $(q+1)$
simplices of dimension $(q-1)$. 

\paragraph{}
The {\bf Whitney complex} $G$ of a graph $(V,E)$ consists of the vertex sets of complete subgraphs of $G$. 
The {\bf Barycentric refinement} $\phi(G)$ of $G$ is the Whitney complex of the graph $\Gamma(G)$ in which $G$ 
are the vertices and where two vertices connected if one is contained in the other. If the maximal dimension of $G$ is $q$, there
is an upper triangular $(q+1) \times (q+1)$ {\bf Barycentric matrix} $A$ such that $f_{\phi(G)} = A f_G$. 
Contractible spaces, varieties, Dehn-Sommerville manifolds, manifolds, spheres or manifolds with boundary
are all invariant under the Barycentric map $\phi$. Other data like the Betti vector stay invariant too.

\paragraph{}
The eigenvectors of the {\bf Barycentric refinement operator} $A f(G) = f(\phi(G))$ - given explicitly as  $A_{ij} = S(i,j) j!$
with {\bf Stirling numbers} $S(i,j)$ of the second kind - 
are $\lambda_k=k!$ for $k=1, \dots, (q+1)$. 
Each eigenvector $w$ of $A^T$ to the eigenvalue $\lambda = k!$ defines a valuation $W: G \to w \cdot f_G$ 
that scales like $W(\phi(G)) = \lambda W(G)$ under Barycentric refinement $\phi$.
Since the eigenvector of $A^T$ to the eigenvalue $1$ is $(1,-1,1,-1,\dots ,(-1)^q)$, 
the {\bf Euler characteristic} $X(G) = f_0-f_1+f_2 - \cdots + (-1)^q f_q = -f_G(-1)$ is immediately recognized
as the only valuation $X$ that is $\phi$-invariant and is normalized $X(1)=1$. 

\paragraph{}
The set $V$ of $0$-dimensional simplices in $G$ can be identified with $\sum_{x \in G} x$. This is
done with the bijection $\{v\} \to v$, identifying the {\bf one point set} $\{ v\}$ with the {\bf point element} $v \in V$ 
it contains. Define the {\bf curvature functional} $K(G)=1-f_0(G)/2+f_1(G)/3 + \cdots - (-1)^q f_q(G)/(q+2)$ of a complex $G$. 
While $K$ is only affine and so not a valuation, the functional $2-2K(G) = 2f_1(G)/3 + \cdots + 2 (-1)^q f_{q}(G)/(q+2)$, 
is a valuation. With $K$, the {\bf Gauss-Bonnet formula} looks $\chi(G) = \sum_{v \in V(G)} K(S(v))$. Using the simplex generating 
function $f_G(t) = 1+f_0 t+ \cdots + f_q t^{q+1}$, we have $\chi(G) = 1-f_G(-1)$. We can also express $K$ 
elegantly as $K(G) = \int_{-1}^0 f_G(t) \; dt$. The Gauss-Bonnet formula 
$f'_G(t) = \sum_{v \in V} f_{S(v)}(t)$ allows to compute $f_G(t)$ recursively. 

\paragraph{}
A {\bf coloring function} $g$ on a complex $G$ is a map $g:V \to K_k = \mathbb{R}$, 
for which $g(v) \neq g(w)$ if $\{v,w\} \in G$. If $G$ is the Whitney complex of a graph, then a vertex
coloring of the graph defines a coloring of $G$. The graph of the Whitney complex of $\Gamma$ - the Barycentric refinement of $\Gamma$-
always has a coloring with $q+1$ colors, if the maximal dimension of $G$ was $q$: just take the dimension functional $g(x)= {\rm dim}(x)$. 
A coloring function defines the integer-valued {\bf divisor function} 
$i_g(v) = 1-\chi(S^-_g(v))$ with $S^-_g(v) = \{ x \in S(v), g(w)<g(v), \forall w \in x \}$ is the
{\bf Poincar\'e-Hopf index}. If $G$ is a Barycentric refinement, then ${\rm dim}(v)$ is a coloring
since $v$ is a simplex in the original complex. In that case $S^-(v)=S^-_{{\rm dim}}(v)$, explaining the notation.

\paragraph{}
The Poincar\'e-Hopf theorem $\sum_{v \in V} i_g(v) = \chi(G)$ generalizes to 
$f_G(t) = 1 + t \sum_{v \in V} f_{S_g^-(v)}(t)$ for the $f$-function, again giving a recursive algorithm to compute
$f$ from generating functions of smaller spaces $S_g^-(v)$. The {\bf symmetric index} at $v \in V$ is 
defined as $j_g(v) = [i_g(v) + i_{-g}(v)]/2$. 
If $G$ is a Barycentric refinement of a complex $H$ and $g={\rm dim}$ is the coloring, then $i_g(v) = \omega(v)
= (-1)^{{\rm dim}(v)}$ and Poincar\'e-Hopf just restates that the Euler characteristic of $G$ is the same than 
the Euler characteristic of $H$. 

\paragraph{}
With respect to the product measure $\mu=\prod_{v \in V} dx$, almost every 
function $g \in \prod_{v \in V} [0,1]$ is a coloring. 
This implies that $\mu$-almost every function $g$ define a $\{-1,1\}$-valued function 
$h(w) = {\rm sign}(g(w)-g(v))$ on $S(v)$, an element in $\Omega = \prod_{w \in V(S(v))} \{-1,1\}$. 
The {\bf push-forward measure} from the coloring space $\prod_{v \in V} [0,1]$ to the spin space
$\Omega$ is the {\bf beta binomial mixture} $P=\int_0^1 {\bf Binom}(p,n) dp$.

\paragraph{}
The measure $P$ on $\Omega$ is also known as the {\it Bayesian predictive distribution 
for Bernoulli trials with a uniform prior on $p$}. We call it simply the {\bf Bayes measure} and
justify it that Bayes introduced it 1763 in Proposition 8 of 
\cite{Bayes1763}. It has the property that the random variable counting the number of $1$ for
$\omega \in \Omega = \{-1,1\}^n$ is uniformly distributed. 
(See footnote in \cite{Bayes1763} after Proposition 9). The measure $P$ is not a product measure any more:
the individual $\omega_k$ are all positively correlated 
${\rm Cor}[\omega_k,\omega_l] = 1/3$ for all $k \neq l$. 

\paragraph{}
The level set $S_{g}(v)$ appears in the {\bf index formula} 
$j_g(v) = 1-\chi(S(v)/2 - \chi(S_g(v)$. While this formula works for any finite abstract simplicial 
complex $G$, the objects $S(v)$ and $S_g(v)$ are Dehn-Sommerville manifolds if $G$ is a Dehn-Sommerville manifold. 
In a manifold case, both $S(v)$ and $S_g(v)$ are manifolds. All of this follows from the {\bf level surface theorem}
that works both in the category of manifolds as well as in the class of Dehn-Sommerville manifolds or
in the class of varieties or in the class of manifolds with boundary. 

\paragraph{}
{\bf Index expectation} dealt with the average of $j_g(v)$ with respect to the measure 
$\mu = \prod_{v \in V(G)} dx$ on $\prod_{v \in V(G)} [0,1]$ that is pushed forward to the
measure $P$ on $\Omega$. Index expectation still gives ${\rm E}[j_g(v)] = K(S(v))$ also, 
if the measure $P$ is used. 

\section{The random manifold theorem}

\paragraph{}
By taking expectation of the index formula we immediately get:

\begin{thm} For every complex $G$, the average Euler characteristic of
            co-dimension-1 complex $H$ in $G$ is ${\rm E}[\chi(H)] = 2-2 K(G)-\chi(G)$. \end{thm}

\begin{proof}
Rearranging the index formula
$$ j_g(v) = 1-\chi(S(v)/2 - \chi(S_g(v))  \; . $$
gives
$$ \chi(S_g(v))  = 2-2 j_g(v) - \chi(S(v))  \; . $$
Take expectation to get 
$$  E[\chi(S_g(v))] = 2 -2 K(v) - \chi(S(v)) \; . $$
Because $S(v)$ can be any complex, just call it $G$. Then $K(v)=K(G)$. Also call $S_g(v)$ the sub-complex $H$.
We read
$$  E[\chi(H)] = 2-2 K(G) - \chi(G) \; . $$
\end{proof} 

\paragraph{}
The step to see the unit sphere $S(v)$ as the actual simplicial complex is just a shift of
perspective. While trivial it was in retrospect the main obstacle to not see this theorem earlier. 
We knew the index formula in 2012 and the level surface theorem in the codimension 1 case  in 2015. 
We also made numerical experiments measuring the average Euler characteristic in manifolds. The
reason for not seeing any patterns was that when taking a Binomial distribution like take a function 
$g$ taking two values with equal probability (a binomial measure), does not lead to interesting results. 

\paragraph{}
For an odd-dimensional Dehn-Sommerville manifold $G$, where $\chi(S(v))=2$ for all $v$,
the index formula gives $K(G) = 1-{\rm E}[\chi(S_h(v)]/2$ so that ${\rm E}[\chi(S_h(v))]=2-2K(G)$. 
It produces the {\bf random manifold theorem}:

\begin{thm} For every odd-dimensional manifold $G$, the average Euler characteristic of
            sub-manifolds $H$ in $G$ is ${\rm E}[\chi(H)] = 2-2 K(G)$. \end{thm}

\paragraph{}
This gives an integral geometric meaning to curvature of higher dimensional manifolds $M$: 

\begin{thm}
If $G$ is a unit sphere of a $(2d+2)$-manifold $M$, then
the Gauss-Bonnet-Chern-Levitt curvature $K(v)$ of $M$ at $v$ is 
$K(v)=1-{\rm E}[\chi(H)]/2$ and so linked to the expected Euler characteristic of 
sub-manifolds $H$ in $S(v)$. 
\end{thm}

\paragraph{}
Given a 4-manifold $G$ for example, if random sub-manifolds in a small $3$-sphere $S_r(v)$ 
have large expected genus, then the curvature is large too. If on the other hand, random 
surfaces in $S_r(v)$ have many connected components that are spheres, then the 
curvature is negative. In a positive curvature 4-manifold (where we know that the 
Gauss-Bonnet-Chern curvature) is then positive too, we see lots of ``interaction" in 
random submanifolds, while in negative curvature situations, the ``world sheets" $H$ 
in $S_r(v)$ are mostly spheres.

\paragraph{}
Functional versions of Gauss-Bonnet, Functional versions of Poincar\'e-Hopf appear because
the result works for all valuations, not only for Euler characteristic $\chi$. The later is 
characterized as the only Barycentric invariant and normalized valuation. Most other valuations
grow exponentially with refinement if they were not valuations that were initially zero. 
Since the index formula is linked to Gauss-Bonnet and Poincar\'e-Hopf, it can also be 
formulated in a functional way. 

\paragraph{}
If 
$$   f_G(t)= 1+f_0(G) t + f_1(G) t^2 + \cdots + f_q(G) t^{q+1} $$ 
is the simplex generating function of $G$, denote by 
$$  K_G(t) = 1 + f_0(G) t/2 + f_1(G) t^2/3 + \cdots f_q(G) t^{q+1}/(q+2) $$
the {\bf curvature generating function}.
The curvature functional $K(G)=1-f_0(G)/2+f_1(G)/3 + \cdots - (-1)^q f_q(G)/(q+2)$ is just 
$K_G(-1)$.  We have $K_G(t) = \int_0^t f_G(s) \; ds$. 

\paragraph{}
Let $e_H$ the $f$-function of the sub-manifold $H=G_g$ using the inherited dimension.
We call it the {\bf inherited $f$ function}. 
This means $e_H(t) = 1+e_0 t^2 + e_1 t^3 + \cdots +  e_k t^{q+1}$. Note that there is no linear term 
in this polynomial, as no vertices can be split by a random function. When we draw the graph of $H$
like in figure 1 above, we fall back into a set-up where we can talk about simplicial complexes. 

\begin{thm}
Given an arbitrary $q$-complex $G$. Random $(q-1)$-complexes in $G$ have 
average inherited f-function ${\rm E}[e_{H}(t)] = 2-2 K_G(t) + f_G(t)$.
\end{thm}

\paragraph{}
Since $-e_H(-t)-1=\chi(H)$, this extends the Euler characteristic result. 
The $f$-function could be obtained from the inherited $f$-function as $f_G(t) =1+(e_G-1)/t$.

\paragraph{}
In the case of varieties, we get the expectation of combinatorial data for $(q-1)$-varieties $H=G_g$, 
for Dehn-Sommerville manifolds we get expectations of combinatorics for $(q-1)$-Dehn-Sommerville manifolds.
For manifolds we get expectations of combinatorics of $(q-1)$-manifolds. 

\paragraph{}
On vectors $\vec{f}_G=(f_0,f_1,\dots, f_q) \to {\rm E}[\vec{e}_H] = ((1-2/3)f_1, (1-2/4) f_2, (1-2/5) f_3,...)$
is linear. For $q=5$ for example, this linear map is given by the $5 \times 6$ matrix: 
$$ \left[
                   \begin{array}{ccccccc}
                   0 & \frac{1}{3} & 0 & 0 & 0 & 0 & 0 \\
                   0 & 0 & \frac{1}{2} & 0 & 0 & 0 & 0 \\
                   0 & 0 & 0 & \frac{3}{5} & 0 & 0 & 0 \\
                   0 & 0 & 0 & 0 & \frac{2}{3} & 0 & 0 \\
                   0 & 0 & 0 & 0 & 0 & \frac{5}{7} & 0 \\
                  \end{array} 
                  \right] \ ; . $$
We know exactly how many $k$-simplices we expect to have. 
The average volume for example is $e_q = f_q (1-2/(q+2))$. 

\paragraph{}
The up-shot is that the expectations of all combinatorial data of submanifolds $H$ are known if the 
combinatorial data of the host manifold $G$ are given. 

\section{Remarks}

\paragraph{}
The random manifold theorem has become a low hanging fruit in an established geometry of simplicial complexes.
One can see it as an example for the ``rising sea metaphor" or ``nut allegory" of Grothendieck. 
Rather general results like Gauss-Bonnet, Poincar\'e-Hopf, index expectation, the index formula and 
the submanifold theorem lead to precise information about the statistics of 
sub-manifolds in a given manifold. Each step is relatively simple, the end result is not obvious at first. 
Here are some references: 

\begin{enumerate}
\item Gauss-Bonnet: $\sum_{v \in V} K(v) = \chi(G)$. 
\cite{Eberhard1891,Levitt1992,elemente11,cherngaussbonnet,dehnsommervillegaussbonnet}
\item Poincar\'e-Hopf: $\sum_{v \in V} i_g(v) = \chi(G)$. 
\cite{Glass1973,poincarehopf,KnillEnergy2020,parametrizedpoincarehopf,MorePoincareHopf,PoincareHopfVectorFields}
\item Index expectation: $E[i_g(v)] = K(v)$. 
\cite{Santalo,KlainRota,indexexpectation}     
\item Index formula: $j_g(v) = 1- \chi(S(v))/2 - \chi(S_g(v))/2$. 
\cite{indexformula,eveneuler,ConstantExpectationCurvature,DiscreteHopf,DiscreteHopf2,GraphProducts} 
\item Hyperbolicity: $S(x)=S^+(x) \oplus S^-(x)$. 
\cite{Unimodularity,CountingMatrix,GreenFunctionsEnergized,EnergizedSimplicialComplexes,EnergizedSimplicialComplexes2, EnergizedSimplicialComplexes3}.
We made here a small adaptation as we have used $U^+(x) = \{ y \in G, y \neq x, x \subset y \}$ before. It is a bit more
elegant to look at the simplicial complex $S^+(x)$ as defined here.
\item Submanifold theorem: $G_g$ is a submanifold of $G$ if not empty.
\cite{KnillSard,ManifoldsPartitions} relate to Morse-Sard \cite{Sard1942,Morse1939,Sard1958}.
\cite{KnillSard,ManifoldsPartitions,DehnSommervilleManifolds}
\item Topology: \cite{Alexandroff1937,May2008,KnillTopology2023}. This especially means to have sensible definitions
of {\bf unit spheres} $S(x)$ as the topological boundary of the smallest open sets $U(x)$. 
\end{enumerate}

\paragraph{}
The results 1)-6) in this list have all first been formulated in the graph setup, which means for simplicial complexes that are 
Whitney complexes of graphs. This was easier to develop, simply because graphs are {\bf more intuitive}
than sets of sets. But everything can be translated into the simplicial complex set-up, once one
gets used to a non-Hausdorff topological space so that one can think about 
unit spheres $S(x)$ in an intuitive way too, allowing to think almost as intuitively as in graphs. 

\paragraph{}
As a side remark, we mention that the results 1)-6) were not appreciated. The reason is notation and cultural. 
We only later realized that many graph theorists saw graphs as one-dimensional simplicial complexes. 
This is reflected in notation: the graph $K_3$ would be considered a
``circle" and girth $3$ would be considered the systole. For a topologist, the graph $K_3$ is contractible of dimension $2$
with trivial fundamental group. For a graph theorist the Betti vector of $K_3$ would be $(1,1)$, while for a topologist it is
$(1,0)$. A triakis icosahedron (which contains $K_4$ graphs) would still be considered 
$2$ dimensional as it is a triangulation of a 2-sphere. The difficulty with definitions has been explained well in \cite{Lakatos} 
or \cite{Gruenbaum2003}. 

\paragraph{}
The discussion of what we should consider to be a homeomorphism on a simplicial complexes was probed in
\cite{KnillTopology} as topological deformations should 
preserve notions like dimension and preserve structures like manifolds. The difficulty is that
the distance metric on graphs does not generate a useful topology as any finite metric space
generates the trivial topology: every subset both open and closed. 
But graphs have already natural unit spheres $S(v)$, the boundary of the neighborhood graph $B(v)$ of $v$. 

\paragraph{}
At the moment, we prefer not to 
use maps on vertices but maps on the topology $\mathcal{O}$: a complex $H$ is a {\bf continuous image} of $G$ if there
exists a map $\phi: \mathcal{O}(G) \to \mathcal{O}(H)$ such that the closure of $\phi(U(x))$ is contractible in $H$
for all $x \in G$. Two complexes $G,H$ are homeomorphic, if $H$ is a continuous image of $G$ and $G$ is a continuous
image of $H$. Such continuous maps are better suited to model {\bf topological dynamical systems}. Simplicial maps are
too rigid and can not model expansive maps like $T(x)=2x \; {\rm mod} \; 1$ on the circle for example. 

\paragraph{}
Integral geometry links geometry and probability theory. It is a tool to compute geometric quantities 
like length, area, volume or curvature using statistics. Our theorem has given an explicit formula 
for the Euler characteristic of a $(2d)$-dimensional submanifold in a given $(2d+1)$-dimensional 
manifold. The formula is exact, not just asymptotic. Since the probability spaces are finite, we can 
in small enough examples generate all possible submanifolds. 
Everything is combinatorial, the probability spaces are finite. The manifold $G$ is a 
finite abstract simplicial complex. It features finitely many sub-manifolds and we compute the average. 

\paragraph{}
In the continuum, questions about {\bf random geometries} are much harder. One has first of all to establish a 
suitable probability space. The set of continuous $1$-dimensional manifolds connecting two points $A,B \in \mathbb{R}^n$
for example leads to the {\bf Wiener measure} on $\{ r \in C([0,1],\mathbb{R}^n), r(0)=A,r(1)=B\}$. 

\paragraph{}
One can take a probability space of Gaussian random functions and look at level sets $\{ f(x)=0 \}$ which 
produce manifolds with probability $1$. An other possibility is to use the Laplacian $L$ on a compact manifold
take a frequency cut-off and use $g=\sum_{i=1}^n a_i \phi_i$, where $a_i$ are IID normal distributed random variables
and $\phi_i$ normalized eigenvectors of the Laplacian. In this context of {\bf random Gaussian fields} one always has to
take suitable limits.

\paragraph{}
An other possibility that has appeared in the Gaussian random field set-up is closer to the {\bf Gaussian splat concept} in 
engineering. Take a random ensemble of points $v_1, \dots, v_n$ in the manifold, take Gaussian random variables $\phi_{j}(x)$ for 
centered at $x_j$. If all these $\phi_j$ are IID, take a multi-variate Gaussian 
$g(x) = \sum_j \phi_j(x)$. Volume rendering techniques in computer graphs could use 4 such functions in a $RGBA$ color model
to represent a 3 dimensional scene. It has become a superior method in comparison with triangulation based models 
\cite{Westover1991}.

\paragraph{}
An other closer analog of the situation done here in the continuum is to compute the expectation of Euler characteristic of a random 
2d-dimensional projective variety in a $(2d+1)$-dimensional 
projective space. If one restricts to polynomials of a certain degree, there are obvious natural measures.
For the computation of the average Euler characteristic of random real algebraic variety, see \cite{Buergisser2007}. 

\paragraph{}
We have seen that the theorem still works, but the level sets are no more manifolds if 
$G$ is no more manifold. In general, any sub-monoid $(\mathcal{G},\oplus)$ of the monoid of all simplicial complexes
with join operator $\oplus$ works and if sub-spheres $S(x)$ are in the same category. 
Let $\mathcal{H}$ be a submonoid of all simplicial complexes $\mathcal{G}$ which has the property 
that if $G \in \mathcal{H}$ an $v \in V(G)$, then $S(v) \in \mathcal{H}$. Such a space $\mathcal{H}$ automatically
defines a monoid. Here are examples:

\begin{itemize}
\item q-manifolds with (or without) boundary. In this case, every unit sphere in the boundary is a manifold
with boundary and every unit sphere inside is a manifold without boundary. Given a $q$-manifold $G$ with boundary $\delta G$
can can look at the $q-2$ manifold, which is the boundary of $G_g$. It is a submanifold of $\delta G$ but
nothing else than $(\delta G)_g$. 
\item The class of {\bf $q$-varieties}, defined as complexes which have the property 
that all $S(x)$ are $(q-1)$-varieties with the base assumption that $0$ is a $(-1)$-variety. 
\item The full class $\mathcal{G}$ of all simplicial complexes. 
\item The class of delta sets (and in particular simplicial sets, which are delta sets with more structure). 
  For a delta set $G=(G_0,G_1,\dots, G_q)$ take $V=G_0$ and use the face maps to define what elements in $G_k$ are 
  contained in $G_l$ for $k<l$, producing a partial order $\subset$. 
  Now define $S(x) = \{ y \in G, y\subset x {\rm or} x \subset y, and x \neq y \}$.
\item The class of contractible complexes does not work because it does not contain $0$. Contractible spaces form only a 
  semi-group and not a monoid with respect to the join operation. 
\end{itemize} 

\paragraph{}
An other possibility to get random complexes is to start with a random graph and look at its Whitney complex.
This produces Erdoes-Renyi simplicial complexes. 
The expectation of the Euler characteristic of such complexes was computed in 
\cite{randomgraph}. On graphs with $n$ vertices, the expected Euler characteristic was
$\sum_{k=1}^n (-1)^k \left( \begin{array}{c} n \\ k \end{array} \right) p^{ \begin{array}{c} k \\ 2 \end{array} }$.
It allowed us to show the existence of connected complexes
of Euler characteristic exponentially large in the number of vertices. 

\paragraph{}
In our case, we can make statements about the existence of manifolds with large Euler characteristic. 
Manifolds are interesting in the context of complexity theory as because of the uniform dimension bound 
on neighborhood graphs, combinatorial data like the f-vector can be computed in polynomial time for manifolds.
This is not the case for non-manifolds. 
If $G$ that is a $q$ manifold and a function $g$ from the vertex set $V$ to 
$\{-1,1\}$, one gets a submanifold $G_g$ of dimension $q-1$, if the level set is not empty. 

\section{Examples}

\paragraph{}
Example: q=1. We have a cyclic sequence of spins $\{-1,1\}$ and $P[x_k=1]=p$,
$P[x_k=-1]=(1-p)$. The number of sign changes is $n p (1-p) + n (1-p) p=2np(1-p)$. 
When integrated from $0$ to $1$ this is $n/3$. Curvature therefore is $1-n/6$. 
This is the curvature functional of $C_n$ which has $f_0=f_1=n$ so 
that $K = 1-f0/2+f_1/3 = 1-n/6$. The equation $f_0=f_1$ is the only 
Dehn-Sommerville relation. To summarize: 
In a 1-manifold $G$ the Euler characteristic expectation of 0-submanifolds is $|E|/3$.

\paragraph{} 
Example: q=3. For 3-manifolds there are two Dehn-Sommerville relations 
$f0-f1+f2-f3=0$ and $f2=2f3$, allowing to reduce everything
$f1,f2$  and produces $K = 1-f1/6+f2/10$. If we sum this $K(S(v))$ over
all $3$ spheres of a 4-manifold, we get the Euler characteristic of that
$4$ manifold. In a 4-manifold of positive Euler characteristic, there are always
many sub-manifolds in unit spheres of positive Euler characteristic.
In a 3-manifold $G$ the Euler characteristic expectation of 2-submanifolds is $|E|/3-|F|/5$.

\paragraph{}
An edge refinement of a 3-manifold takes an edge $(a,b)$ replaces it with a pair $(a,c)(c,b)$.
The circle $S(a,b)=S(a) \cap S(b)$ which was is then connected to $c$ so form the new complex.
In other words, every maximal simplex containing $(a,b)$ is split into two. 
Under edge refinement of 3 manifolds the f-vector changes as follows
$f_0 \to f_0+1, f_1 \to f_1+n+1, f_2 \to f_2+2n,f_3 \to f_3+n$. 
It follows that the expectation increases by $(5-n)/15$ where $n$ is the length of $S(a) \cap S(b)$.
If we nave bone sizes larger and smaller than 5, we can change the
curvature up or down linearly with the number of facets. 

\paragraph{}
By using $(2d+2)$-manifolds with some large curvature we can get
$2d$ manifolds with large Euler characteristic. How large can we make
the curvature in an even-dimensional manifold? 
For connected 2 manifolds the curvature can not be larger than $1/3$:
The equation $1/3=1-{\rm E}[X]/2$ gives ${\rm E}[X]=2-2/3$ but it can be arbitrarily 
negative so that the number of nodes can be arbitrarily large. 
While in the continuum, the GBC curvature can become arbitrarily large, we can 
not build arbitrary small spheres. 

\paragraph{}
We can use Barycentric refinement to build large manifolds in which the 
asymptotic size of the $f$ vector is known. 
Since the eigenvector to the Barycentric operator with the largest
eigenvalue is $(2, 13, 22, 11)$ we have asymptotically $f_2=22 f_1/13$.
which makes the functional go to $f_1/6-f_2/10 \sim -f_1/390$. This means
that for large Barycentrically refined $3$-manifolds, we expect
more components than holes in the surfaces. We have $2-2g \sim -f_1/390$ 
so that we expect for highly refined manifolds to have genus $g \sim C_3 f_1$
with $C_3=1/780$. 

\section{Code}

\paragraph{}
The following code snipped allows to experiment. Take an arbitrary simplicial complex $G$,
then run the Monte Carlo process $M[G]$. It will report the successive averages after each 
batch of 1000 sub simplicial complexes H have been computed. If G is a manifold, then 
the submanifold H is a manifold. Note that take $-\chi(H)$ as $H$ is a priori just 
a {\bf delta set} in which the $0$-dimensional parts are $1$-dimensional simplices in $G$. 

\begin{tiny}  
\lstset{language=Mathematica} \lstset{frameround=fttt}
\begin{lstlisting}[frame=single]
Generate[A_]:=If[A=={},{},Sort[Delete[Union[Sort[Flatten[Map[Subsets,A],1]]],1]]]; 
Whitney[s_]:=Generate[FindClique[s,Infinity,All]];   L=Length; R:=Random[]; T=True;
Surface[G_,g_]:=Select[G,SubsetQ[#/.g,{1,2}] &];  w[x_]:=-(-1)^L[x]; Chi[G_]:=Total[Map[w,G]];   
K[G_]:=1-Sum[x=G[[k]];w[x]/(L[x]+1),{k,L[G]}]; X[G_]:=2-2K[G]-Chi[G]; V[G_]:=Union[Flatten[G]];
M[G_]:=Module[{k=0,t=0,W=V[G],q=X[G],g,p},While[T,k++;p=R;g=Table[W[[j]]->If[R<p,1,2],{j,L[W]}];
  e=-Chi[Surface[G,g]]; t=t+e; If[Mod[k,1000]==0,Print[N[{k,e,q,t/k}]]]]];

G=Whitney[CycleGraph[11]]; M[G];                         (* q=1    *)
G=Whitney[WheelGraph[11]]; M[G];                         (* q=2    *)
G=Whitney[GraphJoin[CycleGraph[9],CycleGraph[7]]]; M[G]; (* q=3    *)
G=Whitney[RandomGraph[{15,50}]];    M[G]                 (* Random *)
\end{lstlisting}
\end{tiny}

\paragraph{}
The following self-contained snippet looks at the functional version. 
We check that $K_G(t) = 1+f_G(t)/2-{\rm E}[f_{S_g}(t)/2]$ so 
that ${\rm E}[f_{S_g}(t)] = 2-2K_G(t)+f_G(t)$. For small random networks, we then
actually sum over the entire {\bf micro canonical ensemble}, meaning to sum 
over all possible sub-manifolds of $G$. Each  element $\omega \in \Omega = \{1,2\}^n$ 
where $n=|V|=f_0(G)$ is weighted with the Bayes weight $p(\omega) = 1/(n+1) B(n,m))$, if 
the sequence $\omega=\{\omega_1, \dots , \omega_n)$ (representing the random function $g:V \to \{1,2\}$)
contains $m$ number of entries $1$.  As Bayes has realized, this measure $P$ is a probability
measure for which the random variable that gives for an element $\omega$ the number of entries $1$
is uniformly distributed. 

\begin{tiny}   
\lstset{language=Mathematica} \lstset{frameround=fttt}
\begin{lstlisting}[frame=single]
Generate[A_]:=If[A=={},{},Sort[Delete[Union[Sort[Flatten[Map[Subsets,A],1]]],1]]];
Whitney[s_]:=Generate[FindClique[s,Infinity,All]]; L=Length; R:=Random[];   T=True; Clear[t]
U[G_,x_]:=Select[G,SubsetQ[#,x] &]; S[G_,x_]:=Complement[u=U[G,x]; Generate[u],u];
f[G_]:=1+Sum[t^L[G[[k]]],{k,L[G]}]; K[G_]:=Expand[Integrate[f[G],{t,0,T}]/T /. T->t];
Surface[G_,g_]:=Select[G,Union[# /. g]=={1,2} &];  j[G_,g_]:=f[Surface[G,g]];
P[x_]:=Module[{m=L[Position[x,1]],n=L[x]},1/(Binomial[n,m]*(n+1))];
J[G_]:=Module[{W=Union[Flatten[G]],A},A=Tuples[{1,2},L[W]]; 
                Sum[g=Table[W[[i]]->A[[m,i]],{i,L[W]}]; P[A[[m]]]*j[G,g],{m,L[A]}]];
H[G_]:=Expand[FullSimplify[1+(f[G]-J[G])/2]];  
G = Whitney[RandomGraph[{10,30}]]; FullSimplify[H[G]==K[G]]
\end{lstlisting}
\end{tiny}

\bibliographystyle{plain}

\begin{thebibliography}{10}

\bibitem{Alexandroff1937}
P.~Alexandroff.
\newblock Diskrete {R\"aume}.
\newblock {\em Mat. Sb. 2}, 2, 1937.

\bibitem{Bayes1763}
F.R.S. Bayes.
\newblock An essay towards solving a problem in the doctrine of chances.
\newblock {\em Philosophical Transactions}, 53:370--418, 1763.

\bibitem{Buergisser2007}
P.~B{\"u}rgisser.
\newblock Average euler characteristic of random real algebraic varieties.
\newblock {\em Comptes Rendus Math{\'e}matique}, 345(9):507--512, 2007.

\bibitem{Eberhard1891}
V.~Eberhard.
\newblock {\em Morphologie der Polyeder}.
\newblock Teubner Verlag, 1891.

\bibitem{Glass1973}
L.~Glass.
\newblock A combinatorial analog of the {P}oincare index theorem.
\newblock {\em Journal of combinatorial theory}, 15:264--268, 1973.

\bibitem{Gruenbaum2003}
B.~Gr{\"u}nbaum.
\newblock Are your polyhedra the same as my polyhedra?
\newblock In {\em Discrete and computational geometry}, volume~25 of {\em
  Algorithms Combin.}, pages 461--488. Springer, Berlin, 2003.

\bibitem{KlainRota}
D.A. Klain and G-C. Rota.
\newblock {\em Introduction to geometric probability}.
\newblock Lezioni Lincee. Accademia nazionale dei lincei, 1997.

\bibitem{randomgraph}
O.~Knill.
\newblock The dimension and {Euler} characteristic of random graphs.
\newblock {\\}http://arxiv.org/abs/1112.5749, 2011.

\bibitem{cherngaussbonnet}
O.~Knill.
\newblock A graph theoretical {Gauss-Bonnet-Chern} theorem.
\newblock {\\}http://arxiv.org/abs/1111.5395, 2011.

\bibitem{elemente11}
O.~Knill.
\newblock A discrete {Gauss-Bonnet} type theorem.
\newblock {\em {Elemente der Mathematik}}, 67:1--17, 2012.

\bibitem{poincarehopf}
O.~Knill.
\newblock A graph theoretical {Poincar\'e-Hopf} theorem.
\newblock {\\} http://arxiv.org/abs/1201.1162, 2012.

\bibitem{indexformula}
O.~Knill.
\newblock An index formula for simple graphs \hfill.
\newblock {\\}http://arxiv.org/abs/1205.0306, 2012.

\bibitem{indexexpectation}
O.~Knill.
\newblock On index expectation and curvature for networks.
\newblock {\\}http://arxiv.org/abs/1202.4514, 2012.

\bibitem{eveneuler}
O.~Knill.
\newblock The {E}uler characteristic of an even-dimensional graph.
\newblock {{\\}http://arxiv.org/abs/1307.3809}, 2013.

\bibitem{KnillTopology}
O.~Knill.
\newblock A notion of graph homeomorphism.
\newblock {{\\}http://arxiv.org/abs/1401.2819}, 2014.

\bibitem{KnillSard}
O.~Knill.
\newblock A {S}ard theorem for graph theory.
\newblock {http://arxiv.org/abs/1508.05657}, 2015.

\bibitem{Unimodularity}
O.~Knill.
\newblock On {F}redholm determinants in topology.
\newblock {\\}https://arxiv.org/abs/1612.08229, 2016.

\bibitem{ConstantExpectationCurvature}
O.~Knill.
\newblock Constant index expectation curvature for graphs or {R}iemannian
  manifolds.
\newblock https://arxiv.org/abs/1912.11315, 2019.

\bibitem{CountingMatrix}
O.~Knill.
\newblock The counting matrix of a simplicial complex.
\newblock {\\}https://arxiv.org/abs/1907.09092, 2019.

\bibitem{dehnsommervillegaussbonnet}
O.~Knill.
\newblock Dehn-{S}ommerville from {G}auss-{B}onnet.
\newblock {\\}https://arxiv.org/abs/1905.04831, 2019.

\bibitem{EnergizedSimplicialComplexes}
O.~Knill.
\newblock Energized simplicial complexes.
\newblock https://arxiv.org/abs/1908.06563, 2019.

\bibitem{MorePoincareHopf}
O.~Knill.
\newblock {More on Poincar\'e-Hopf and Gauss-Bonnet}.
\newblock https://arxiv.org/abs/1912.00577, 2019.

\bibitem{parametrizedpoincarehopf}
O.~Knill.
\newblock A parametrized {P}oincare-{H}opf theorem and clique cardinalities of
  graphs.
\newblock {\\}https://arxiv.org/abs/1906.06611, 2019.

\bibitem{PoincareHopfVectorFields}
O.~Knill.
\newblock Poincar{\'e}-{H}opf for vector fields on graphs.
\newblock {\\}https://arxiv.org/abs/1911.04208, 2019.

\bibitem{EnergizedSimplicialComplexes2}
O.~Knill.
\newblock Division algebra valued energized simplicial complexes.
\newblock https://arxiv.org/abs/2008.10176, 2020.

\bibitem{KnillEnergy2020}
O.~Knill.
\newblock The energy of a simplicial complex.
\newblock {\em Linear Algebra and its Applications}, 600:96--129, 2020.

\bibitem{GreenFunctionsEnergized}
O.~Knill.
\newblock Green functions of energized complexes.
\newblock {\\}https://arxiv.org/abs/2010.09152, 2020.

\bibitem{EnergizedSimplicialComplexes3}
O.~Knill.
\newblock Green functions of energized complexes.
\newblock https://arxiv.org/abs/2010.09152, 2020.

\bibitem{DiscreteHopf}
O.~Knill.
\newblock Integral geometric {H}opf conjectures.
\newblock https://arxiv.org/abs/2001.01398, 2020.

\bibitem{DiscreteHopf2}
O.~Knill.
\newblock On index expectation curvature for manifolds.
\newblock https://arxiv.org/abs/2001.06925, 2020.

\bibitem{GraphProducts}
O.~Knill.
\newblock The curvature of graph products.
\newblock https://arxiv.org/abs/2107.08563, 2021.

\bibitem{KnillTopology2023}
O.~Knill.
\newblock Finite topologies for finite geometries.
\newblock {{\\}http://arxiv.org/abs/2301.03156}, 2023.

\bibitem{ManifoldsPartitions}
O.~Knill.
\newblock Manifolds from partitions.
\newblock https://arxiv.org/abs/2401.07435, 2024.

\bibitem{DehnSommervilleManifolds}
O.~Knill.
\newblock Dehn sommerville manifolds.
\newblock https://arxiv.org/abs/2508.14372, 2025.

\bibitem{Lakatos}
I.~Lakatos.
\newblock {\em Proofs and Refutations}.
\newblock Cambridge University Press, 1976.

\bibitem{Levitt1992}
N.~Levitt.
\newblock The {E}uler characteristic is the unique locally determined numerical
  homotopy invariant of finite complexes.
\newblock {\em Discrete Comput. Geom.}, 7:59--67, 1992.

\bibitem{May2008}
J.P. May.
\newblock Finite topological spaces.
\newblock Notes for REU, Chicago, 2003-2008, 2008.

\bibitem{Morse1939}
A.P. Morse.
\newblock The behavior of a function on its critical set.
\newblock {\em Ann. of Math. (2)}, 40(1):62--70, 1939.

\bibitem{Santalo}
L.A. Santalo.
\newblock {\em Introduction to integral geometry}.
\newblock Hermann and Editeurs, Paris, 1953.

\bibitem{Sard1942}
A.~Sard.
\newblock The measure of the critical values of differentiable maps.
\newblock {\em Bull. Amer. Math. Soc.}, 48:883--890, 1942.

\bibitem{Sard1958}
A.~Sard.
\newblock Images of critical sets.
\newblock {\em Ann. of Math. (2)}, 68:247--259, 1958.

\bibitem{Westover1991}
L.A. Westover.
\newblock Splatting: a parallel feed forward volume rendering algorithm.
\newblock Dissertation at the university of North Carolina at Chapel Hill,
  1991.

\end{thebibliography}

\end{document}